\documentclass[11pt]{article}

\usepackage{amsthm}
\usepackage{amsmath}
\usepackage{amssymb}
\usepackage{bbm}
\usepackage{cancel}
\usepackage{enumitem}
\usepackage[top=2.5cm,bottom=2.5cm]{geometry}
\usepackage{graphicx}
\usepackage{mathrsfs}
\usepackage{mathtools}
\usepackage{microtype}
\usepackage[indent]{parskip}
\usepackage{setspace}
\usepackage{soul}
\usepackage{tikz-cd}
\usepackage{xcolor}

\usepackage%[backref=page]
    {hyperref} % always at end

\definecolor{primary}{HTML}{348aa7} % 4d6c81 for something slightly less apparent...

\definecolor{accent1}{HTML}{5dd39e} % large
\definecolor{accent2}{HTML}{348aa7} % ring
\definecolor{accent3}{HTML}{bce784} % small

\allowdisplaybreaks

\hypersetup{
	linkcolor  = primary,
	urlcolor   = primary,
	citecolor  = primary,
	colorlinks = true
}  

\newcommand{\defeq}{\mathrel{\overset{\makebox[0pt]{\mbox{\normalfont\tiny\sffamily
 def}}}{=}}}
\newcommand{\key}[1]{\emph{{#1}}}

\DeclareMathOperator{\id}{id}
\DeclareMathOperator{\Dens}{Dens}

\newtheorem{theorem}{Theorem}[section]
\newtheorem{lemma}[theorem]{Lemma}
\newtheorem{proposition}[theorem]{Proposition}
\newtheorem{corollary}[theorem]{Corollary}

\theoremstyle{definition}
\newtheorem{definition}[theorem]{Definition}
\newtheorem{example}[theorem]{Example}
\newtheorem{remark}[theorem]{Remark}

\title{Morse-Bott Volume Forms}
\date{}

\author{Luke Volk\thanks{
    Department of Mathematics,
    University of Toronto;
    e-mail: \texttt{luke.volk@mail.utoronto.ca}}~
     and Boris Khesin\thanks{
    Department of Mathematics,
    University of Toronto;
    e-mail: \texttt{khesin@math.toronto.edu}}
}

\begin{document}
    \maketitle

    \begin{abstract}
        A Morse-Bott volume form on a manifold is a top-degree form which vanishes along a non-degenerate critical submanifold. We prove that two such forms are diffeomorphic (by a diffeomorphism fixed on the submanifold) provided that their relative cohomology classes with respect to the submanifold coincide. For a zero submanifold of codimension at least 2, this means that two Morse-Bott volume forms with the same zero set are diffeomorphic if and only if they have equal total volumes. We show how ``Moser’s trick'' for establishing equivalence of non-degenerate volume forms can be adapted to this setting.
    \end{abstract}

    \section{Introduction}

    \subsection*{Background}

        In his 1965 paper \cite{moser1965volume}, Moser showed that any two volume forms $\eta_0,\eta_1$ with the same total volume on a compact, connected, and oriented manifold $M$ are related by a diffeomorphism $\Phi$ of $M$ via pullback, $\Phi^*\eta_1 = \eta_0$. To construct such a diffeomorphism, Moser's method was to connect the forms $\eta_0$ and $\eta_1$ by a path $(\eta_t)_{t\in[0,1]}$ in the same cohomology class and to look for a family of diffeomorphisms $(\Phi_t)_{t\in[0,1]}$ such that $\Phi_0 = \id_M$ and $\Phi_t^*\eta_t = \eta_0$. The latter is achieved by solving the corresponding infinitesimal version of the equation on the vector field generating the flow $\Phi_t$ and invoking the existence theorem for ODEs guaranteeing the existence of a flow for a given vector field, verifying the conditions for its existence for all $t\in[0,1]$. The strategy has dubbed variably as ``Moser's trick'', ``Moser's path method'', or the ``homotopy method''.

        This method has seen a wide variety of applications. Moser also applied the method in \cite{moser1965volume} to symplectic structures and twisted volume forms on non-orientable manifolds. Banyaga described Moser's approach for volume forms on manifolds with boundary in \cite{banyaga1974formes}, while Bruveris \emph{et al.} \cite{Bruveris2016MosersTO} extended it to volume forms on manifolds with corners. Cardona and Miranda \cite{cardona2019volume} considered an analogue of Moser's result for equivalence of top-degree forms transverse to the zero section with a shared zero hypersurface. Other authors have considered solutions to the so-called ``pullback equation'' $\Phi^*\eta_1 = \eta_0$ in more analytic contexts, see e.g. a summary of equivalence results for $k$-forms for any $k$ for H\"older spaces in \cite{csato2011pullback}.

    \subsection*{Main result}

        Let $M$ be a compact connected oriented $n$-dimensional manifold, which we equip with a reference (non-vanishing) volume form $\mu$. In this paper, we consider volume forms on $M$ which have a quadratic degeneration along an oriented submanifold $\Gamma\subset M$.    
        
        \begin{definition}
            A \key{Morse-Bott volume form} for $\Gamma$ on $M$ is a non-negative $n$-form $\eta\in\Omega^n(M)$ with zero set $\Gamma$ such that the ratio of $n$-forms $f \defeq \eta/\mu$ is a Morse-Bott function $f\colon M\to\mathbb R$ for which each component of $\Gamma$ is a non-degenerate critical submanifold. 
        \end{definition}

        Note that the critical zero set $\Gamma$ must have Morse-Bott index $0$ since the function $f$  is non-negative on $M$. Furthermore, the Morse-Bott property of $\eta$ does not depend on choice of the reference form $\mu$. We prove the necessary and sufficient conditions for diffeomorphism equivalence of such Morse-Bott volume forms:
        
        \begin{theorem}\label{th_morsebottmoser}
            Let $\eta_0$ and  $\eta_1$ be Morse-Bott volume forms for $\Gamma\subset M$  such that their relative cohomology classes with respect to $\Gamma$ coincide: 
                \[[\eta_0]=[\eta_1]\in H^n(M,\Gamma)\,.\]
            Then there exists a diffeomorphism $\Phi\colon M\to M$ such that $\Phi^*\eta_1 = \eta_0$ which restricts to the identity on $\Gamma$.
        \end{theorem}

        We treat this as two different cases: when the submanifold $\Gamma$ is a hypersurface, and when its codimension is at least 2.
       
        \begin{corollary}
            If the shared zero submanifold $\Gamma\subset M$ of two Morse-Bott volume forms $\eta_0$ and $\eta_1$ on $M$ is of codimension at least 2, the forms are diffeomorphic,
                \[\Phi^*\eta_1 = \eta_0 \text{ with } \Phi|_\Gamma = \id_\Gamma,\]
            if and only if they have equal total volumes of $M$,
                \[\int_M\eta_0 = \int_M\eta_1.\] 
        \end{corollary}
        
        If $\Gamma$ is a hypersurface in $M$, i.e. it has codimension $1$, it can be separating or not. Either case is covered by the following corollary:
  
        \begin{corollary}
            If the shared zero submanifold $\Gamma\subset M$ has codimension $=1$, two Morse-Bott volume forms $\eta_0$ and $\eta_1$ are diffeomorphic,
                \[\Phi^*\eta_1 = \eta_0 \text{ with } \Phi|_\Gamma = \id_\Gamma,\]
            if and only if they have coinciding volumes for each connected component $M_i$ of $M\setminus\Gamma$: 
                \[\int_{M_i}\eta_0 = \int_{M_i}\eta_1 \quad\text{ for all }i.\]
        \end{corollary}
  
        The same result also holds for volume forms which have hypersurface $\Gamma\subset M$ as a  non-critical zero set. Let $\eta_0$ and $\eta_1$ be two $n$-forms on $M$ with the same \key{non-critical zero set} $\Gamma$, i.e. it is a non-critical zero set for each of the corresponding functions $\eta_i/\mu$, $i=0,1$. Note that $\Gamma$ must be a compact oriented hypersurface in this case. 
 
        \begin{theorem}\label{th_cardonamirandamoser}
            Two $n$-forms $\eta_0$ and $\eta_1$ with the same non-critical zero set $\Gamma\subset M$ are diffeomorphic,
                \[\Phi^*\eta_1 = \eta_0 \text{ with } \Phi|_\Gamma = \id_\Gamma,\]
            if and only if they represent the same relative cohomology classes $[\eta_0]=[\eta_1]\in H^n(M,\Gamma)$, or equivalently, they have coinciding volumes of each connected component $M_i$ of $M\setminus \Gamma$.
        \end{theorem}

        This theorem strengthens one of the results of Cardona and Miranda \cite{cardona2019volume}, who proved that two folded volumes forms with the same non-critical zero hypersurface $\Gamma\subset M$ can be mapped to each other by a diffeomorphism taking $\Gamma$ to itself, although not necessarily the identity on $\Gamma$. (Note that while $n$-forms change sign across the non-critical zero set $\Gamma\subset M$, the term \emph{folded}, or \emph{transversally vanishing}, ``volume forms'' became standard and we adapt it in this paper.)

        \begin{remark}
        The assumption of the orientability of $M$ and $\Gamma$ can be weakened to require only 
        orientability of $M\setminus \Gamma$. For ${\rm codim}\,\Gamma\ge 2$ this reduces to orientability of $M$. In the case of ${\rm codim}\,\Gamma=1$ in a nonorientable $M$, instead of volume forms one needs to consider densities, or pseudo-forms changing sign along orientation-reversing paths. 
        Theorem \ref{th_morsebottmoser} naturally extends to this setting: its proof combines tubular neighbourhood embeddings, which hold in any setting, with adapting the classical Moser theorem to each (orientable) connected component of $M\setminus \Gamma$. (As an example, it is easy to construct a Morse-Bott density on the even-dimensional real projective space with a hyperplane as a critical set: for instance take the product of function $x_1^2$ with the standard volume element on the sphere $S^{2k}$ and project it to $\mathbb {RP}^{2k}$ via the antipodal map.)
        An extension of Theorem  \ref{th_cardonamirandamoser} to the nonorientable setting might be more subtle, requiring certain averaging on the orientation cover, cf. normal forms for  Morse functions and densities in the $n=2$ area-preserving case in \cite{izosimov2025}.
        \end{remark}

    \subsection*{Motivation}
    
        A motivation for this problem comes from the Madelung transform, which establishes an equivalence of quantum mechanics and equations of compressible fluids \cite{khesin2019geometry}. Namely, let a wave function $\psi\colon M\to\mathbb C$ on a manifold $M$ satisfy the non-linear Schr\"odinger (NLS) equation,
            \[{i}\partial_t\psi +\Delta\psi +  V\psi + f(|\psi|^2)\psi=0,\]
        where $V\colon M\to \mathbb R$ and $f\colon \mathbb R_{+}\to \mathbb R$. Then the \key{Madelung transform} $\psi=\sqrt{\rho e^{i\theta}}$ allows one to rewrite the quantum mechanics of the NLS equation in a ``hydrodynamical form'' as equations of a barotropic-type fluid on the velocity field $v \defeq \nabla\theta$ and the density $\rho$ as follows:
            \[
            \begin{cases} 
                \partial_t v + \nabla_v v + \nabla\Big(V + f(\rho) - \frac{\Delta\sqrt{\rho}}{\sqrt{\rho}} \Big) = 0 \\ 
                \partial_t\rho +{\rm div}(\rho v) = 0\,.
            \end{cases}
            \]
        The Madelung transform $\psi\mapsto (\rho, \theta)$ is well-defined provided that $\psi$ does not vanish on $M$ and it is understood modulo a phase factor ($\psi\sim \psi e^{i\alpha}$), while $\theta$ is understood to be modulo an additive constant on $M$. Moreover, by confining to the unit sphere of normalized wave functions $\psi$ and the space $\Dens(M)$ of normalized densities $\rho$, the Madelung transform can be understood as the map $\mathbb{CP}(M,\mathbb C\setminus 0) \to T^*\Dens(M)$. It turns out to be a symplectomorphism for the corresponding natural symplectic structures on those spaces, and a K\"ahler map between the Fubini-Study and Fisher-Rao metrics respectively, see \cite{khesin2019geometry}.

        However, the presence of zeros of the (complex-valued) wave function $\psi$ brings substantial complications. A non-critical zero set $\Gamma$ of $\psi$ has codimension 2 in $M$, and the corresponding density function $\rho$ can be understood as a Morse-Bott volume form for $\Gamma\subset M$. The fact that $\psi$ is univalued on $M$ imposes the ``quantization constraint'' on the phase function $\theta$: its change along any path in $M$ going around $\Gamma$ must be a multiple of $4\pi$, see numerous discussions in \cite{fritsche2009stochastic,wallstrom1994inequivalence}. The above equivalence theorems for the Morse-Bott volume forms allow one to deal with zero submanifolds of wave functions by using  more convenient ``normal forms'' of the corresponding densities around zeros.

    \subsection*{Acknowledgements}
    
        We are indebted to Alexander Givental for key suggestions on the proof, and to Yael Karshon for fruitful discussions.  We are also grateful to the anonymous referee for useful suggestions. B.K. was partially supported by an NSERC Discovery Grant.

    \section{The Morse-Bott Lemma}

        Morse-Bott functions have local normal forms which will allow us to more easily handle behaviour near the zero set of Morse-Bott volume forms. A (``local'') normal form in a neighbourhood of a point can be found, e.g., in \cite{banhurtub2004proof}. Below we outline a proof of a (``semi-global'') normal form in a neighbourhood of the critical set $\Gamma$ using Euler-like vector fields, following \cite{meinrenken2021euler}.

    \subsection*{Euler-Like vector fields}

        Given a submanifold $\Gamma\subseteq M$ of codimension $k$, we denote the \key{normal bundle} of $\Gamma$ in $M$ by:
            \[\nu(M,\Gamma) \defeq TM|_\Gamma/T\Gamma.\]
        Morphisms between pairs $(M,\Gamma)$ and $(M',\Gamma')$ are smooth maps $f\colon M\to M'$ taking $\Gamma$ to $\Gamma'$. Given a morphism $f\colon (M,\Gamma)\to (M',\Gamma')$, we associate to it the linear map $\nu(f)$ defined as follows:
            \begin{align*}
                \nu(f)\colon \nu(M,\Gamma) &\to \nu(M',\Gamma') \\
                v+T\Gamma &\mapsto f_*v + T\Gamma',
            \end{align*}
        which we call the \key{linearisation} of $f$.

        The \key{Euler vector field to $\Gamma$} is the vector field $\mathcal E$ on the normal bundle $\nu(M,\Gamma)$ which is the Euler vector field in the usual sense on each fibre. That is, if $x\in \Gamma$ and the fibre $\nu(M,\Gamma)_x$ is given the coordinates $y^i$, then:
            \[\mathcal E_x = \sum_{i=1}^{k} y^i\frac{\partial}{\partial y^i}.\]
        If a vector field $X$ on $M$ is tangent to $\Gamma$ (i.e., for each $p\in \Gamma$, $X_p\in T_p\Gamma$) then $X$ can be seen as a morphism $X\colon(M,\Gamma)\to(TM,T\Gamma)$ of pairs. We say that $X$ is \key{Euler-like} for $\Gamma$ if its linearisation,
            \[\nu (X)\colon \nu(M,\Gamma)\to \nu(TM,T\Gamma) \cong T\nu(M,\Gamma),\]
        is the Euler vector field to $\Gamma$.

        \begin{example}
            For $M=\mathbb R^n$ and $\Gamma = \{0\}$, we have that $\nu(M,\Gamma) = T_0\mathbb R^n$ and a vector field 
                \[X = \sum_{i=1}^nX^i\frac{\partial}{\partial x^i}\]
            is Euler-like if for all $1\leq i\leq n$ we have $X^i(0) = 0$ and $DX^i|_0 = x^i$.
        \end{example}

    \subsection*{Tubular neighbourhood embeddings}

        A \key{tubular neighbourhood embedding} of $\Gamma$ is a neighbourhood $U\subseteq\nu(M, \Gamma)$ of the zero section in the normal bundle and an embedding $\varphi\colon U\to M$ such that:
            \begin{enumerate}[label=(\roman*)]
                \item For each $x\in \Gamma$, $\varphi(0_x) = x$. That is, $\varphi|_\Gamma = \id_\Gamma$ after identifying $\Gamma$ with the zero section in $\nu(M,\Gamma)$.
                \item The linearisation $\nu(\varphi)\colon\nu(U,\Gamma)\cong \nu(M,\Gamma)\to\nu(M,\Gamma)$ is the identity, $\id_{\nu(M,\Gamma)}$.
            \end{enumerate}
        The benefit of Euler-like vector fields is their correspondence with tubular neighbourhood embeddings as the following theorem summarizes:
            
        \begin{theorem}\label{th_tubularnbhd}
            An Euler-like vector field $X$ for $(M,\Gamma)$ determines a unique maximal (with respect to inclusion) tubular neighbourhood embedding $\varphi\colon U \to M$ of $\Gamma$ with $U\subseteq\nu(M,\Gamma)$ such that $\varphi^*X = \mathcal E$.
        \end{theorem}
        
        \noindent We refer for the proof to  \cite{meinbursztyn2019splitting}.

    \subsection*{Fibre-wise polynomial functions}

        We say that $f\colon M\to\mathbb R$ is \key{Morse-Bott for $\Gamma\subset M$} if $\Gamma$ is a non-degenerate critical submanifold of $f$. Without loss of generality, we assume $f|_\Gamma = 0$.  The Euler vector field to $\Gamma$ gives a handy method for identifying fibre-wise homogeneous polynomials on $\nu(M,\Gamma)$.
            
     %   The ideal $\mathfrak m$ induces a filtration:
      %      \[C^\infty(M) \defeq \mathfrak m^0 \supseteq \mathfrak m \supseteq \mathfrak m^2 \supseteq \mathfrak m^3 \supseteq \dotsc,\]
     %   where $\mathfrak m^j$ is the set of functions on $M$ which vanish on $\Gamma$ with order $k$. 
     
     %   From the filtration $(\mathfrak m^j)_{j=0}^\infty$, we define a graded algebra:
          %  \[\mathcal P_\Gamma \defeq \bigoplus_{j=0}^\infty \mathfrak m^j/\mathfrak m^{j+1}.\]
      %  An element in $\mathcal P_\Gamma$, which is a sum of finitely many non-zero elements from the summands of $\mathcal P_\Gamma$, is called a \key{fibre-wise polynomial function} on $\nu(M,\Gamma)$. 

        \begin{proposition}\label{prop_eulerhomogeneous}
            Let $\mathcal E$ be the Euler vector field to $\Gamma$ and $f\in C^\infty(\nu(M,\Gamma))$ a function on $\nu(M,\Gamma)$. If $\mathcal L_{\mathcal E}f = kf$ for some $k\in\mathbb N$, then $f$ is fibre-wise a homogeneous polynomial of order $k$.
        \end{proposition}
        \begin{proof}
            This is a fibre-wise application of Euler's homogeneous function theorem.
        \end{proof}

        The following proof of the Morse-Bott lemma (as sketched in \cite{meinrenken2021euler}) makes use of Euler-like vector fields, and can be regarded as a semi-global, fibre-wise version of the Morse lemma with parameters.
        
        \begin{theorem}[Morse-Bott lemma~{\cite{meinrenken2021euler}}]\label{th_MorseBottLemma}
            If $f\colon M\to\mathbb R$ is Morse-Bott for $(M,\Gamma)$ and $f|_\Gamma = 0$, then there exists a tubular neighbourhood embedding $\varphi\colon U \to M$ (with $U \subseteq\nu(M,\Gamma)$) such that $\varphi^* f$ is fibre-wise a homogeneous polynomial of degree 2 (i.e. fibre-wise quadratic).
        \end{theorem}
        \begin{proof}
            Without loss of generality, take $M$ to be a tubular neighbourhood of $\Gamma$, so $M$ sits inside $\nu(M,\Gamma)$. Because $f$ is Morse-Bott, for each $x\in \Gamma$, the functions
                \[g_x \defeq f|_{\nu(M,\Gamma)_x \cap M}\colon \nu(M,\Gamma)_x\cap M \to \mathbb R\]
            are \emph{Morse} functions, each with a non-degenerate critical point at $0_x$, where the Hessian $Hg_x|_{0_x}$ is non-degenerate. In coordinates $y^i$ on the fibre $\nu(M,\Gamma)_x$ near $0_x$ (we suppress the subscript $x$),
                \[g(y) = \frac12\sum_{i,j} A_{ij}(y)y^iy^j,\]
            where $A(y) = (A_{ij}(y))$ is a symmetric matrix-valued function such that $A(0) = Hg|_{0_x}$. We compute:
                \[\frac{\partial g}{\partial y^k}
                    = \left(\frac12\sum_{i,j} y^iy^j\frac{\partial A_{ij}}{\partial y^k}\right) +\sum_{i}A_{ik}y^i = \sum_i B_{ki}y^i,\]
            here $B_{ki} = A_{ik} + \frac12\sum_j\frac{\partial A_{ij}}{\partial y^k}y^j$. Note that $Hg|_{0_x} = A(0) = B(0)$, and so the matrix-valued function $B$ is invertible in a neighbourhood of $y=0$. We will now construct an Euler-like vector field on $\nu(M,\Gamma)_x$, analogous to the construction in the proof of the Morse lemma in \cite{meinrenken2021euler}. Let $X$ be a vector field (implicitly depending on $x$, more accurately notated $X_x$) in this neighbourhood near zero by:
                \[X = \sum_{i,j}(AB^{-1})_{ij}(y)\,y^i\frac{\partial}{\partial y^j}.\]
             Since $A(0)B^{-1}(0) = I$, near zero $\nu(X) = \mathcal E$, and so $X$ is Euler-like. Then we have:
                \[X(g) = \sum_{i,j} (AB^{-1})_{ij}(y) y^i\frac{\partial g}{\partial y^j} = \sum_{i,j,k} (AB^{-1})_{ij}(y)B_{jk}(y)y^iy^k = \sum_{i,k}(AB^{-1}B)_{ik}(y)y^iy^k = 2g.\]
            Now define a vector field $Y$ on $M$ by $Y(x,y) = X_x(y)$. By construction, $Y$ is Euler-like and, by Theorem~\ref{th_tubularnbhd}, $Y$ determines a tubular neighbourhood embedding $\varphi\colon U\to M$ such that $\varphi^*Y = \mathcal E$. Note that for each $x\in \Gamma$, the vector field $X_x$ was defined to be tangent to the fibre at $x$, and so $Y$ only flows along the fibres, its flow being $\phi_t(x,y) = (x,\phi_t^{X_x}(y))$, where $\phi_t^{X_x}$ is the flow of $X_x$. We then can compute that:
                \[(\mathcal L_Yf)(x,y) = \left.\frac{d}{dt}\right|_{t=0} f(x,\phi_t^{X_x}(y)) = \left.\frac{d}{dt}\right|_{t=0}g_x(\phi_t^{X_x}(y)) = (\mathcal L_{X_x}g_x)(y) = 2g_x(y) = 2f(x,y),\]
            hence applying $\varphi^*$ to both sides yields $\mathcal L_{\mathcal E}(\varphi^*f) = 2\varphi^*f$. By Proposition~\ref{prop_eulerhomogeneous}, $\varphi^*f$ must be fibre-wise a homogeneous polynomial of degree 2.
        \end{proof}

        In our case, where the Morse-Bott functions associated with Morse-Bott volume forms necessarily have index 0 (i.e. correspond to positive-definite quadratic forms), we have a convenient normal form:

        \begin{theorem}
            If $f\colon M\to\mathbb R$ is Morse-Bott for $\Gamma\subset M$ of index 0 and $f|_\Gamma=0$, then there exist coordinates $(x,y)$ in the tubular neighbourhood $U$ of Theorem~\ref{th_MorseBottLemma} ($x$ parametrising $\Gamma$, and $y$ the fibres) such that:
                \[(\varphi^*f)(x,y) = |y|^2.\]
        \end{theorem}
        \begin{proof}
            This theorem is equivalent to the existence of \emph{bundle metrics} (also known as \emph{Euclidean metrics}, see \cite{jost2008riemannian}) on the normal bundle. It is based on the partition of unity and the fact that the space of positive definite quadratic forms in $n$-variables is a convex cone. This allows one to combine the diffeomorphism $\varphi$ constructed in Theorem~\ref{th_MorseBottLemma} taking $f$ to its quadratic part with a  fibre-wise linear map $L$, so that the composition $L\circ \varphi$   takes $f$ to the fibre-wise quadratic function given by the length-squared in the fibre.
        \end{proof}
        
        \begin{corollary}\label{cor_MorseBottindex0}
            Suppose that $f_0$ and $f_1$ are Morse-Bott functions for $\Gamma\subset M$, both with a Morse-Bott index of 0. Then there exists a neighbourhood $U$ of $\Gamma$ and a diffeomorphism $\varphi$ defined on $U$ such that $\varphi^*f_1 = f_0$.
        \end{corollary}

        The analogue of the corollary for a maximal Morse-Bott index $k = \text{codim}(\Gamma)$ follows similarly. A generalization of this result for any index is proven for fibre-wise quadratic functions on general vector bundles in \cite{dossena2013sylvester} subject to the constraint that the positive- and negative-definite parts of $f_0$ and $f_1$ give the same splittings of the vector bundle.

        \begin{example}
            Note that the existence of the universal semi-global normal form in Corollary~\ref{cor_MorseBottindex0} for Morse-Bott functions of index 0 is based on the contractibility of the cone of symmetric positive definite matrices, allowing one to connect any two such functions in a tubular neighborhood of their critical set and apply Moser's trick (discussed in more detail in the next section in the context of forms). Morse-Bott functions of non-max/minimal indices might be non-isotopic, as can be seen in the following example.
            
            The space of non-degenerate quadratic forms $ax^2 + 2bxy + cy^2$ in two variables is split into three components by the double cone $b^2 - ac = 0$ in the 3D space $(a,b,c)\in\mathbb R^3$, see Figure~\ref{fig_cones}. This allows the following simple construction of a pair of non-isotopic Morse-Bott functions, as they realize contractible and non-contractible loops in the set of forms of index~1.

            Let $M=S^2\times S^1\subset \mathbb R^3 \times S^1$ be a 3-manifold regarded as a bundle over $S^1 = \{\theta \mod 2\pi\}$, where $S^2 = \{(x,y,z)\in\mathbb R^3 \mid x^2 + y^2 + z^2 = 1\}$.  As one of the functions one can take the restriction $g|_{S^2\times S^1}$ of the function $g(x,y, z, \theta) = x^2 - y^2$ from $\mathbb R^3\times S^1$ to $M$, independent of the variables $z$ and $\theta$. Its critical set $\Gamma \defeq \{(x,y,z)\in\mathbb R^3\mid x=y=0, z=\pm 1\}\subset M$ consists of two (north and south) circles,  both having the Morse-Bott index 1.
      
            The other function is also defined by the restriction $f|_{S^2\times S^1}$ of a fibre-wise quadratic $f$ in $\mathbb R^3 \times S^1$, where: 
                \[f(x,y,z,\theta) =  (x^2 - y^2)\cos\theta + 2xy\sin\theta.\]
            For each $\theta$, the restriction of $f(\cdot,\theta)$ to the sphere has non-degenerate critical points of index 1 at the north and south poles of the fibre $S^2$ and it defines a non-contractible loop in the space of quadratic forms, as illustrated in Figure \ref{fig_cones}. As a result, there is no isotopy between $f$ and $g$.
        \end{example}

            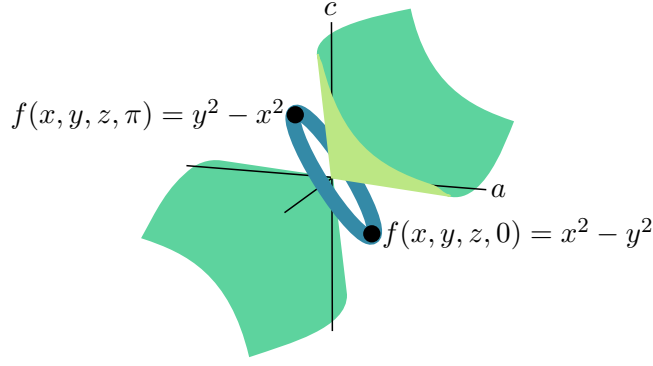
\begin{figure}[ht]
                \centering
                
                \begin{tikzpicture}[scale=0.25]
                    \begin{scope}[shift={(-0.3146, 0.8614)}]
                        \path[draw=accent1,fill=accent1,line width=0.0265cm] (11.0171, 17.2961) -- (11.749, 11.1231).. controls (11.7808, 9.3934) and (8.774, 8.5907) .. (6.6642, 7.8776).. controls (5.6201, 12.2251) and (3.3664, 13.3546) .. (0.9823, 14.1374).. controls (1.7058, 15.4294) and (2.2721, 16.7777) .. (3.6788, 17.8258).. controls (3.8329, 17.9312) and (4.4106, 18.2778) .. (5.1619, 18.1628) -- cycle;
                
                        \path[draw=accent1,fill=accent1,nonzero rule,line width=0.0265cm] (10.4529, 23.2989).. controls (11.6008, 18.3503) and (14.482, 17.5948) .. (16.316, 16.8637).. controls (18.6616, 14.8229) and (20.1244, 19.1923) .. (20.5721, 20.2618).. controls (18.2387, 21.033) and (16.0075, 22.2272) .. (15.1311, 26.193).. controls (13.5953, 25.8648) and (9.3065, 25.0961) .. (10.4529, 23.2989) -- cycle;

                        \begin{scope}[line width=0.02cm]
                            \path[draw=black] (3.3054, 17.946) -- (19.1729, 16.6734);
                            \path[draw=black] (11.0268, 9.1915) -- (10.9803, 25.5386);
                            \path[draw=black] (8.494, 15.4567) -- (11.0171, 17.2961);
                        \end{scope}
                    
                        \path[draw=accent2,line width=0.2cm,rotate around={-57.9109:(0.0, 29.7)}] (16.2619, 32.6062) ellipse (3.8679cm and 0.7763cm);
                        \path[draw=accent3,fill=accent3,line width=0.0265cm] (10.2575, 23.8515) -- (11.0075, 17.3154) -- (17.2672, 16.3427).. controls (16.8265, 16.4355) and (16.6896, 16.7674) .. (16.3137, 16.8884).. controls (13.2382, 17.8788) and (11.4733, 19.3011) .. (10.4483, 23.2943).. controls (10.4163, 23.4189) and (10.3034, 23.5614) .. (10.2575, 23.8515) -- cycle;
                    \end{scope}

                    \node at (19.5,17.5){$a$};
                    \node at (10.6,27.1){$c$};

                    \fill (12.8,15.2) circle (13pt) node[right]{$f(x,y,z,0) = x^2 - y^2$};
                    \fill (8.75,21.5) circle (13pt) node[left]{$f(x,y,z,\pi) = y^2 - x^2$};
                \end{tikzpicture}
                
                \caption{Example of a non-trivial loop in the space of quadratic forms $ax^2 + 2bxy + cy^2$, as parametrized by $(a,b,c)$.}\label{fig_cones}
            \end{figure}

        \section{Proofs of main results}

        In this section, we will prove Theorems~\ref{th_morsebottmoser} and \ref{th_cardonamirandamoser} with their corollaries.

        \subsection*{Proof of Theorem~\ref{th_morsebottmoser} on Morse-Bott volume forms}\label{pf_morsebottmoser}

        % We begin by recalling the definition of relative cohomology as covered in \cite{bott2013differential}. Given an embedding $f\colon\Gamma\hookrightarrow M$, we can define a co-chain complex $\Omega^\bullet(f) \defeq \bigoplus_{i=0}^\infty\Omega^i(f)$  by:
        %     \[\Omega^i(f) \defeq \Omega^i(M)\oplus\Omega^{i-1}(\Gamma),\qquad d(\omega,\theta) \defeq (d\omega,f^*\omega - d\theta).\]
        % The corresponding homology of $\Omega^\bullet(f)$ is called the \key{relative cohomology} of $M$, relative to $\Gamma$. By cohomology class of a $k$-form $\omega\in\Omega^k(M)$ relative to $\Gamma$, we mean the cohomology class of $(\omega,0)\in\Omega^k(M)\oplus\Omega^{k-1}(\Gamma)$. Note that this inclusion $\omega\mapsto (\omega,0)$ provides an isomorphism between the alternative formulation of relative cohomology \emph{\`a la} Godbillon~\cite{godbillonalgebrique} and that of Bott and Tu.

    % \renewcommand{\thetheorem}{\ref{th_morsebottmoser}}
    %     \begin{theorem}\label{th_morsebottmoser_proof}
    %         Let $\eta_0$ and  $\eta_1$ be Morse-Bott volume forms for $\Gamma\subset M$  such that their relative cohomology classes with respect to $\Gamma$ coincide: 
    %             \[[\eta_0]=[\eta_1]\in H^n(M,\Gamma)\,.\]
    %         Then there exists a diffeomorphism $\Phi\colon M\to M$ such that $\Phi^*\eta_1 = \eta_0$ which restricts to the identity on $\Gamma$.
    %     \end{theorem}

        Assume that $\eta_0$ and $\eta_1$ are Morse-Bott volume forms for $\Gamma\subset M$ such that their relative cohomology classes with respect to $\Gamma$ coincide, $[\eta_0] = [\eta_1]\in H^n(M,\Gamma)$. We will show that there exists a diffeomorphism $\Phi\colon M\to M$ such that $\Phi^*\eta_1 = \eta_0$ which restricts to the identity on $\Gamma$.

        \begin{proof}
            First consider the local problem. Let $N$ be a tubular neighbourhood of $\Gamma$, which we identify with a neighbourhood in the normal bundle, $\nu(M,\Gamma)\subseteq M$. Two Morse-Bott volume forms in $N$ can be expressed as $\rho_0 \defeq f\mu$ to $\rho_1 \defeq h\mu$, where $f$ and $h$ are Morse-Bott functions having $\Gamma$ as a non-degenerate minimum (i.e. a critical submanifold of index 0) and $\mu$ is a reference (non-vanishing) volume form on $N\subset M$. The forms $\rho_i$ can be thought of as the restrictions $\rho_i \defeq \eta_i|_N$ of globally defined Morse-Bott forms $\eta_i$ to the neighbourhood $N$.
        
            By Corollary~{\ref{cor_MorseBottindex0}}, there exists a diffeomorphism $F$ (of a possibly smaller neighbourhood of $\Gamma$) taking $h$ to $f$, but changing the reference form $\mu$. So without loss of generality we assume that, after application of $F\colon N\to N$ the Morse-Bott volume forms are $\rho_0 = f\mu$ and $\rho_1 = f\phi\,\mu$ for some non-vanishing function $\phi\in C^\infty(N)$. Next, the function $f$ can be assumed fibre-wise quadratic in $N\subset \nu(M,\Gamma)$ by Theorem~\ref{th_MorseBottLemma}. We are looking for a diffeomorphism of $N$ pulling back $\rho_1$ to $\rho_0$ while remaining the identity on $\Gamma$. 
        
            To apply Moser's trick, we consider the interpolation $\rho_t \defeq (1-t)\rho_0 + t\rho_1$ of these forms and seek a family of diffeomorphisms $(\psi_t)_{t\in[0,1]}$ such that $\psi_t^*\rho_t = \rho_0$. Applying $\psi_{t*}\frac{d}{dt}$ to this relation, we get that:
                \[\mathcal L_{X_t}\rho_t +\dot\rho_t =0,\]
            where $(X_t)_{t\in[0,1]}$ is the time-dependent vector field whose flow is $(\psi_t)_{t\in[0,1]}$, and $\dot\rho_t=\rho_1 - \rho_0$. We will show that there exists such a smooth vector field $X_t$ vanishing on $\Gamma$ for $t\in [0,1]$.
            
            Let $\mathcal E$ be the Euler vector field to $\Gamma$, defined on $N\subset \nu(M,\Gamma)$, and let $g_s$ be the flow of $-\mathcal E$ towards $\Gamma$. The following expression for a primitive for $\dot\rho_t$ is similar to the one in the proof of the Poincar\'e lemma:
                \begin{align*}
                    \dot\rho_t
                    &= -\int_0^\infty \frac{d}{ds}g_s^*(\dot\rho_t)\,ds
                    % &= -\int_0^\infty g_s^*\left(-\mathcal L_{\mathcal E}\dot\rho_t + \cancel{\frac{d}{ds}\dot\rho_t}\right)\,ds \\
                    = -\int_0^\infty g_s^*\left(-\mathcal L_{\mathcal E}\dot\rho_t\right)\,ds \\
                    &= \int_0^\infty g_s^*\left(d\,\iota_{\mathcal E}\dot\rho_t + \cancel{\iota_{\mathcal E}d\dot\rho_t}\right)\,ds 
                    = \int_0^\infty d\,\iota_{g_s^*\mathcal E}(g_s^*\dot\rho_t)\,ds. %\\
                \end{align*}
            Now note that $g_s^*f = fe^{-2s}$ since $f$ is fibre-wise quadratic, while $\dot\rho_t = \rho_1 - \rho_0 = f(\phi\mu - \mu)$, hence:
                \begin{align*}
                    \dot\rho_t = \int_0^\infty d\left(\iota_{g_s^*\mathcal E}(fe^{-2s}g_s^*(\phi\mu - \mu))\right)\,ds
                    = d\,[f\underbrace{\int_0^\infty\iota_{g_s^*\mathcal E}(e^{-2s}g_s^*(\phi\mu - \mu))\,ds}_{-\beta}],
                \end{align*}
            and therefore, locally near $\Gamma$, one has $-\dot\rho_t = d(f\beta)$ for some ($n-1$)-form $\beta$ on $N$. Note that for all $p\in\Gamma$, we have that $g_s^*\mathcal E|_p = 0$, and so $\beta|_\Gamma = 0$. 

            On the other hand, $\mathcal L_{X_t}\rho_t= d\,\iota_{X_t}\rho_t$, while $\rho_t = f((1-t)\mu + t\phi\mu)$. Hence to solve the equation $\mathcal L_{X_t}\rho_t +\dot\rho_t =0$ for the field $X_t$ or, equivalently, $d\,\iota_{X_t}\rho_t = d(f\beta)$, it suffices to solve: 
                \[f\,\iota_{X_t}((1-t)\mu + t\phi\mu) = f\beta.\]
            This amounts to solving the equation $\iota_{X_t}((1-t)\mu + t\phi\mu) = \beta$ for a family of vector fields $X_t$ on $N$. Note that the volume form $(1-t)\mu + t\phi\mu$ interpolates between $\mu$ and $\phi\mu$ and hence it is non-vanishing for all $t\in [0,1]$. Hence the field $X_t$ solving $\iota_{X_t}((1-t)\mu + t\phi\mu) = \beta$ exists on $N\setminus\Gamma$, and it is smooth and must vanish on $\Gamma$. Note also that due to this vanishing condition, solutions starting sufficiently close to $\Gamma$ exist for the whole interval $t\in [0,1]$. Hence the time-1 map of the flow $\psi_t$ corresponding to the vector field $X_t$ provides the required diffeomorphism of some neighbourhood of $\Gamma$. Without loss of generality, we can assume that this is the neighbourhood $N\supset \Gamma$, and this completes the proof of the local statement.

            To prove the existence of a smooth  globally-defined field on $M$ whose flow takes $\eta_1$ to $\eta_0$, we first  (smoothly and arbitrarily) extend the field $X_t$ from $N$ to the whole of $M$. Now consider a smaller tubular neighbourhood $U$ of $\Gamma$, sitting compactly within $N$, $\Gamma\subset U\subset \overline U\subset N\subset M$, and pick a bump function $b\colon M\to[0,1]$ which is $1$ on $U$ and $0$ on $M\setminus N$. This allows one to define the time-dependent vector field $Y_t \defeq bX_t$ on $M$ whose time-1 flow map $G\colon M\to M$ satisfies $G^*\eta_1|_U = \eta_0|_U$ and $G|_\Gamma = \id_\Gamma$.

            Consider the pull-back action of the map $G$ on the Morse-Bott form $\eta_1$: it is a new form $\zeta_1 \defeq G^*\eta_1$ which coincides with $\eta_0$ in the neighbourhood $U\supset \Gamma$, but outside of $U$, the form $\zeta_1$ is only known to be non-vanishing and, by assumption, representing the same relative cohomology class in $H^n(M,\Gamma)$ as the form $\eta_0$.

            We will now apply Moser's method again to find a diffeomorphism mapping $\zeta_1$ to $\eta_0$ everywhere on $M$. For this we  consider the  interpolation $\zeta_t \defeq (1-t)\eta_0 + t\zeta_1$ between them, joining $\zeta_0 \defeq \eta_0$ and $\zeta_1$. Note that all these forms coincide in the tubular neighbourhood $U$, $\zeta_t|_U=\zeta_0|_U$ for all $t\in [0,1]$. We will be seeking a family of diffeomorphisms $(\varphi_t)_{t\in[0,1]}$ such that $\varphi_t^*\zeta_t = \zeta_0$. Applying $\varphi_{t*}\frac{d}{dt}$ to this relation, we get that:
                \[ \mathcal L_{Z_t}\zeta_t = \zeta_0 - \zeta_1,\]
            where $(Z_t)_{t\in[0,1]}$ is the time-dependent vector field whose flow is $(\varphi_t)_{t\in[0,1]}$. Note that $\zeta_0$ and $\zeta_1$ represent the same class in $H^n(M,\Gamma)$ and $(\zeta_0 - \zeta_1)|_U = 0$. We wish to find a primitive $(n-1)$-form for $\zeta_0 - \zeta_1$ which is zero on $U$. This can be done in a number of ways, cf. \cite{Avila2010, Bruveris2016MosersTO, cardona2019volume, csato2011pullback}, for instance via the following consideration.

            Since $\Gamma$ is a deformation retract of its tubular neighbourhood $U$, the forms $\zeta_0$ and $\zeta_1$ represent the same relative cohomology class in $H^n(M,U)=H^n(M,\Gamma)$. By the definition of relative cohomology, there exists a $\omega\in \Omega^{n-1}(M)$ and $\theta\in\Omega^{n-2}(U)$ such that:
                \[\zeta_0 - \zeta_1 = d\omega,\quad i^*\omega = d\theta,\]
            for the inclusion $i\colon U\hookrightarrow M$. Pick a bump function $\tilde b\colon M\to[0,1]$ equal to $1$ on a smaller tubular neighbourhood $\tilde U$ compactly contained in $U$ and equal to $0$ on $M\setminus U$. Then define:      
                \[\tilde\omega \defeq \omega - d(\tilde{b}\theta) \in \Omega^{n-1}(M),\]
            where $d\tilde\omega = d\omega = \zeta_0 - \zeta_1$ and $\tilde\omega|_{\tilde U} = 0$. To complete Moser's trick, we now want to solve the equation  $\iota_{Z_t}\zeta_t = \tilde\omega$ for an unknown vector field $Z_t$. The Morse-Bott form $\zeta_t|_{M\setminus\Gamma}$ is a volume form for all $t\in[0,1]$ (with $\zeta_t|_U = \zeta_0|_U$) and so $Z_t$ has a unique solution on $M\setminus\Gamma$. Furthermore, the solution $Z_t$ vanishes on the tubular neighbourhood $\tilde U\supset \Gamma$, since $\tilde\omega|_{\tilde U} = 0$. This allows one to define $Z_t$ on the whole manifold $M$ (extending it by zero to $\Gamma$ itself). The corresponding flow of $Z_t$ on the compact manifold $M$ exists for $t\in [0,1]$, and it is the identity on $\tilde U\supset \Gamma$.
            
            Thus Moser's trick yields that the time-1 flow map $H\colon M\to M$ is a diffeomorphism satisfying $H^*\zeta_1 = \zeta_0$ and $H|_\Gamma = \id_\Gamma$. Finally, we define the diffeomorphism $\Phi$ of $M$ as the composition $\Phi \defeq H\circ G\circ F$, where $F$ (extended to $M$) is from the Morse-Bott lemma (Corollary~{\ref{cor_MorseBottindex0}}), $G$ identifies the Morse-Bott volume forms in a neighbourhood of $\Gamma$, and $H$ relates the forms outside of a neighbourhood of $\Gamma$ while keeping fixed the neighbourhood itself. The composition satisfies $\Phi^*\eta_1 = \eta_0$ and $\Phi|_{\Gamma} = \id_\Gamma$, as desired.
        \end{proof}

        Note that the last part of the proof boils down to construction of a vector field with a prescribed divergence, while controlling its support outside of a neighbourhood of the critical set. This topic has a long history. In the $C^\infty$-case for manifolds it was considered in \cite{Avila2010}. The arguments above can be regarded as an extension of the boundary case in \cite{banyaga1974formes, Bruveris2016MosersTO}.

        We now turn to the corollaries of this theorem, which manifest differently depending on the codimension of the critical submanifold $\Gamma$.
        
    \renewcommand{\thetheorem}{1.3}
        \begin{corollary}
            If the shared zero submanifold $\Gamma\subset M$ of two Morse-Bott volume forms $\eta_0$ and $\eta_1$ on $M$ is of codimension at least 2, the forms are diffeomorphic, $\Phi^*\eta_1 = \eta_0$ with $\Phi|_\Gamma = \id_\Gamma$, if and only if they have equal total volumes of $M$,
                \[\int_M\eta_0 = \int_M\eta_1.\] 
        \end{corollary}
        \begin{proof}
            If $\Gamma\subset M$ has codimension $k\geq 2$, the condition of $[\eta_0] = [\eta_1]\in H^n(M,\Gamma)$ is equivalent to the condition that they have equal total volumes. This can be seen from the long exact sequence for relative cohomology:
                \[\cdots\to H^{n-1}(\Gamma) \to H^n(M,\Gamma)\to H^n(M)\to H^n(\Gamma)\to 0,\]
            where the constraint on the codimension implies $H^{n-1}(\Gamma) = H^n(\Gamma) = 0$. This implies an isomorphism $H^n(M,\Gamma) \cong H^n(M)$ by exactness of the sequence:
                \[0\to H^n(M,\Gamma) \to H^n(M)\to 0.\]
        \end{proof}

    \renewcommand{\thetheorem}{1.4}
        \begin{corollary}
            If the shared zero submanifold $\Gamma\subset M$ has codimension 1, two Morse-Bott volume forms $\eta_0$ and $\eta_1$ are diffeomorphic, $\Phi^*\eta_1 = \eta_0$ with $\Phi|_\Gamma = \id_\Gamma$, if and only if they have coinciding volumes for each connected component $M_i$ of $M\setminus\Gamma$: 
                \[\int_{M_i}\eta_0 = \int_{M_i}\eta_1 \quad\text{ for all }i.\]
        \end{corollary}
        \begin{proof}
            If $\Gamma\subset M$ has codimension $k=1$, one has only a surjection $H^n(M,\Gamma)\to H^n(M)$ from $H^n(\Gamma) = 0$. In general, it is not necessarily an isomorphism, due to the possible disconnectedness of $M\setminus\Gamma$. If $M\setminus\Gamma$ consists of several connected components, $M\setminus\Gamma = \bigsqcup_{i\in I} M_i$, the forms $\eta_0$ and $\eta_1$ represent the same cohomology class relative to $\Gamma$ if and only they have equal volumes on each component $M_i$. Indeed, by definition of relative de Rham cohomology, the forms $\eta_0$ and $\eta_1$ represent the same cohomology class relative to $\Gamma$ whenever their difference is exact, $d\omega = \eta_1 - \eta_0$ for some $\omega\in\Omega^{n-1}(M)$. Hence we have:
                \[\int_{M_i}(\eta_1-\eta_0) = \int_{M_i}d\omega = \int_{\partial M_i}\omega = 0\,,\]
            where the last equality follows from the fact that $\omega$ is exact on $\partial M_i\subset\Gamma$.
        \end{proof}

    \subsection*{Non-critical zero sets}

    \renewcommand{\thetheorem}{\thesection.\arabic{theorem}}
    \setcounter{theorem}{0}
    
        A \key{folded volume form} on an oriented $n$-dimensional manifold $M$ is a top-degree form $\eta\in \Omega^n(M)$ which is transverse to the zero section of the determinant bundle $\Lambda^nT^*M$. Here we outline how the following strengthening of the result of Cardona and Miranda~\cite{cardona2019volume} on the equivalence of two folded volume forms with the same zero sets
        can be proven using a similar strategy through Euler-like vector fields.
        
        Note that the zero set $\Gamma\subset M$ of a folded volume form $\eta$ is an oriented hypersurface $\Gamma$ (possibly consisting of several components). By fixing a reference volume form, $\mu$, the hypersurface $\Gamma$ has a \key{defining function} $\eta/\mu \defeq f\colon M\to\mathbb R$ satisfying $\Gamma = f^{-1}(0)$ and $df|_x \neq 0$ for each $x\in \Gamma$. We have the following analogue of the Morse-Bott lemma (Theorem~\ref{th_MorseBottLemma}):
        
        \begin{lemma}\label{lem_definingfunctionhomdeg1}
            If $f\colon M\to\mathbb R$ is a defining function for a hypersurface $\Gamma$, then there exists a tubular neighbourhood embedding $\varphi\colon U\to M$ (with $U\subseteq\nu(M,\Gamma)$) such that $\varphi^*f$ is fibre-wise linear.
        \end{lemma}
        \begin{proof}
            This is achieved via the inverse function theorem by taking $f$ as the coordinate in the fibres of the one-dimensional normal bundle. In a sense, this is Hadamard's lemma depending on the parameter $x\in \Gamma$. The Euler-like vector field of Theorem~\ref{th_MorseBottLemma} becomes 
            $X = f\frac{\partial}{\partial f}$ in this coordinate.
        \end{proof}

        Now we prove Theorem~\ref{th_cardonamirandamoser}, that two folded volume forms $\eta_0$ and $\eta_1$ on a compact oriented manifold $M$ with the same non-critical zero set $\Gamma\subset M$ are diffeomorphic, $\Phi^*\eta_1 = \eta_0$ with $ \Phi|_\Gamma = \id_\Gamma$, if and only if they represent the same relative cohomology classes $[\eta_0]=[\eta_1]\in H^n(M,\Gamma)$, or equivalently, they have coinciding volumes of each connected component $M_i$ of $M\setminus \Gamma$. The lemma above allows one to set up Moser's trick, much like how Corollary~\ref{cor_MorseBottindex0} was used in the proof of Theorem~\ref{th_morsebottmoser}.

        \begin{proof}
            Consider the defining functions $f_0$ and $f_1$ of $\Gamma$ corresponding to the folded volume forms $\eta_0$ and $\eta_1$ for a fixed volume form $\mu$, i.e. $f_i = \eta_i/\mu$. Hadamard's lemma guarantees the existence of a non-vanishing $\phi\in C^\infty(M)$ such that $\eta_1 = \phi\,\eta_0$.
        
            Taking a tubular neighbourhood $N\subset M$ of $\Gamma$ (and identifying it with a neighborhood in $\nu(M,\Gamma)$), one comes to  a setting analogous to the {proof} of Theorem~\ref{th_morsebottmoser}: here we are constructing a diffeomorphism equivalence between two folded volume forms $\eta_0 = f\mu$ and $\eta_1 =  f\phi\,\mu$ in the neighbourhood $N$, where $\mu$ is a volume form, $\phi$ is a non-vanishing function, and $f$ is a defining function for $\Gamma$. Without loss of generality, we can assume  $f$ to be fibre-wise linear (otherwise applying Lemma~\ref{lem_definingfunctionhomdeg1} to make it so). 

            Now we are looking for a diffeomorphism of a neighbourhood $N\subseteq M$ of $\Gamma$ pulling back $\eta_1$ to $\eta_0$ which is the identity on $\Gamma$. The rest of the proof using Moser's trick follows \emph{mutatis mutandis} the {proof} of Theorem~\ref{th_morsebottmoser}, except that now the function $f$ is fibre-wise \emph{linear}, and hence the pullback $g_s^*f$ of $f$ by the inverse flow $g_s$ of the Euler vector field to $\Gamma$ will be $fe^{-s}$, with the factor of 2 replaced by 1.

            As before, the diffeomorphism of the neighbourhood $N$ extends to a global diffeomorphism $G\colon M\to M$. Then the existence of a diffeomorphism $H$ relating the form on $M\setminus N$ is based on the equality of the corresponding relative cohomology classes $[\eta_0] = [\eta_1]\in H^n(M,\Gamma)$, or equivalently, on the equality of the volumes of connected components $M_i$ of $M\setminus N$. The desired diffeomorphism $\Phi$ is obtained by composing the corresponding diffeomorphisms $G$ and $H$ as in the proof of Theorem~\ref{th_morsebottmoser} (with the application of Hadamard's lemma instead if the diffeomorphism $F$).
        \end{proof}

Note that due to the local-to-global nature of the construction of the diffeomorphism, it immediately follows from the above that if $\eta_0$ and $\eta_1$ are top-degree forms with a shared zero set $\Gamma$ such that each component of $\Gamma$ is either Morse-Bott or non-critical (i.e. $\Gamma$ is of mixed-type) and such that their
relative cohomology classes in $H^n(M\setminus \Gamma)$ coincide, then there is a diffeomorphism $\Phi\colon M\to M$ such that $\Phi^*\eta_1 = \eta_0$ restricting to the identity on $\Gamma$.

    \bibliographystyle{plain}
    % \nocite{*}
    \bibliography{bibliography}

\begin{thebibliography}{10}

\bibitem{Avila2010}
Artur Avila.
\newblock On the regularization of conservative maps.
\newblock {\em Acta Mathematica}, 205(1):5--18, 2010.

\bibitem{banyaga1974formes}
Augustin Banyaga.
\newblock Formes-volume sur les vari{\'e}t{\'e}s {\`a} bord.
\newblock {\em Enseignement Math}, 20(2):127--131, 1974.

\bibitem{banhurtub2004proof}
Augustin Banyaga and David~E. Hurtubise.
\newblock A proof of the {M}orse-{B}ott lemma.
\newblock {\em Expositiones Mathematicae}, 22(4):365--373, 2004.

\bibitem{Bruveris2016MosersTO}
Martins Bruveris, Peter~W. Michor, Adam Parusiński, and Armin Rainer.
\newblock Moser’s theorem on manifolds with corners.
\newblock {\em Proc. Amer. Math. Soc.}, 146(11):4889--4897, 2018.

\bibitem{meinbursztyn2019splitting}
Henrique Bursztyn, Hudson Lima, and Eckhard Meinrenken.
\newblock Splitting theorems for {P}oisson and related structures.
\newblock {\em Journal f{\"u}r die reine und angewandte Mathematik (Crelles Journal)}, 2019(754):281--312, 2019.

\bibitem{cardona2019volume}
Robert Cardona and Eva Miranda.
\newblock On the volume elements of a manifold with transverse zeroes.
\newblock {\em Regular and Chaotic Dynamics}, 24:187--197, 2019.

\bibitem{csato2011pullback}
Gyula Csat{\'o}, Bernard Dacorogna, and Olivier Kneuss.
\newblock {\em The pullback equation for differential forms}, volume~83.
\newblock Springer Science \& Business Media, 2011.

\bibitem{dossena2013sylvester}
Giacomo Dossena.
\newblock Sylvester's law of inertia for quadratic forms on vector bundles.
\newblock {\em preprint arXiv:1307.2171}, 2013.

\bibitem{fritsche2009stochastic}
L.~Fritsche and M.~Haugk.
\newblock Stochastic foundation of quantum mechanics and the origin of particle spin.
\newblock {\em preprint arXiv:0912.3442}, 2009.

\bibitem{izosimov2025}
Anton Izosimov, Boris Khesin, and Ilia Kirillov.
\newblock Coadjoint orbits of area-preserving diffeomorphisms of non-orientable surfaces.
\newblock {\em J. of Symplectic Geometry}, 23(1):1--35, 2025.

\bibitem{jost2008riemannian}
J{\"u}rgen Jost.
\newblock {\em Riemannian geometry and geometric analysis}, volume~4.
\newblock Springer, 2005.

\bibitem{khesin2019geometry}
Boris Khesin, Gerard Misio{\l}ek, and Klas Modin.
\newblock Geometry of the {M}adelung transform.
\newblock {\em Archive for Rational Mechanics and Analysis}, 234:549--573, 2019.

\bibitem{meinrenken2021euler}
Eckhard Meinrenken.
\newblock Euler-like vector fields, normal forms, and isotropic embeddings.
\newblock {\em Indagationes Mathematicae}, 32(1):224--245, 2021.

\bibitem{moser1965volume}
J{\"u}rgen Moser.
\newblock On the volume elements on a manifold.
\newblock {\em Transactions of the American Mathematical Society}, 120(2):286--294, 1965.

\bibitem{wallstrom1994inequivalence}
Timothy~C. Wallstrom.
\newblock Inequivalence between the {S}chr\"odinger equation and the {M}adelung hydrodynamic equations.
\newblock {\em Phys. Rev. A}, 49:1613--1617, Mar 1994.

\end{thebibliography}
\end{document}